\documentclass[twoside,11pt]{amsart}
\usepackage{a4wide}
\usepackage[utf8]{inputenc}
\usepackage[english]{babel}
\usepackage[fixlanguage]{babelbib}
\usepackage{color,graphicx}
\usepackage{amssymb,amsfonts,amsmath,amsthm,amscd, epsfig}
\usepackage[all]{xy}
\usepackage{tensor}
\usepackage{microtype}
\usepackage{comment}
\usepackage{tikz-cd}
\usepackage{float}

\title{Galois Correspondence for Partial Groupoid Actions}
\author[Lautenschlaeger and Tamusiunas]{Wesley G. Lautenschlaeger and Thaísa Tamusiunas}
\address{Instituto de Matem\'{a}tica, Universidade Federal do Rio Grande do Sul,  Av. Bento Gon\c{c}alves, 9500, 91509-900. Porto Alegre-RS, Brazil}
\email{wesley\_gl@hotmail.com}
\email{thaisa.tamusiunas@gmail.com}
\date{}

\usepackage{graphicx}
\usepackage{amssymb}
\usepackage{amsmath}
\usepackage{amsthm}
\usepackage{indentfirst}
\usepackage{tikz}
\usetikzlibrary{cd}

\newcounter{contador}
\numberwithin{contador}{section}

\newtheorem{theorem}[contador]{Theorem}
\newtheorem{prop}[contador]{Proposition}
\newtheorem{lemma}[contador]{Lemma}
\newtheorem{corollary}[contador]{Corollary}

\theoremstyle{definition}
\newtheorem{defi}[contador]{Definition}
\newtheorem{obs}[contador]{Remark}
\newtheorem{exe}[contador]{Example}

\newcommand{\G}{\mathcal{G}}
\newcommand{\cH}{\mathcal{H}}

\newcommand{\cA}{\mathcal{A}}

\begin{document}

\begin{abstract}
   We prove a Galois correspondence theorem for groupoids acting orthogonally and partially on commutative rings. We also consider partial actions that are not orthogonal, presenting two correspondences in this case: one for strongly Galois partial groupoid actions and one for global groupoid actions (without restriction). Some examples are presented.
\end{abstract}

\maketitle

\vspace{0.5 cm}

\noindent \textbf{2020 AMS Subject Classification:} Primary. 13B05. Secondary. 20L05, 18B40.

\noindent \textbf{Keywords:} groupoid, Galois theory, partial actions, Galois correspondence, orthogonal action

\section{Introduction}

Chase, Harrison and Rosenberg developed, in 1965, a Galois Theory for finite groups acting on commutative rings, in where they exhibit a generalization of the Fundamental Theorem of the Galois theory \cite{chase1969galois}. 
Right after, Villamayor and Zelinsky \cite{villamayor1966galois} generalized one of the constructions of \cite{chase1969galois} to extensions of commutative rings with finitely many idempotents. Although they did not explicitly defined groupoid actions, this was the intermediate object that they worked to prove their correspondence theorem. This was the first time that groupoids appeared in a Galois context. 

A Galois correspondence theorem for groupoids acting on commutative rings was showed in \cite{paques2018galois}. Given $B \subseteq A$ a commutative Galois extension with finite Galois groupoid $\G$, it was proved a one-to-one correspondence between the wide subgroupoids of $\G$ and the separable, $\beta$-strong $B$-subalgebras of $A$, where $\beta$ is a (global) action of $\G$ on $A$. Other Galois theories concerning global action of groupoids were also developed in \cite{cortes2017characterisation}, \cite{garta2024}, \cite{lata2021galois}, \cite{pata}, \cite{pataIII} and \cite{pedrotti2023injectivity}.

The study of partial group actions is recent in mathematical terms and its first objective was to classify $C^*$-algebras \cite{exel1998partial}. In a groupoid context, many questions concerning partial actions were investigated. For instance, in \cite{bagio2010partial} it was introduced the concept of partial action of ordered groupoids on rings; in \cite{lau2021semigrupoides}, it was showed that there is a biunivocal correspondence between partial groupoid actions and global inverse semigroupoid actions; in \cite{bagio2012partial}, a globalization theorem was proved and a characterization theorem for Galois extensions was showed. We highlight the special case in \cite{bagio2020restriction}, where the authors related a class of partial groupoid actions to partial group actions, calling them as \emph{group-type} partial groupoid actions. The relation groupoid-group presented in  \cite{bagio2020restriction} was used in \cite{bagio2022galois} to construct a Galois correspondence for partial actions of groupoids, but only restricted to group-type partial actions, and the mechanisms used there cannot be applied to any partial action. Therefore, a different approach needs to be thought of.

So, the first objective of this work is to extend the Galois correspondence presented in \cite{bagio2022galois} to a partial groupoid action without requiring the group-type property. For this, we will relate a partial Galois extension to the Galois extension regarding its globalization. The partial-to-global relation will give a more general partial Galois correspondence, as we will present in Section 3.

The second objective is to improve the assumption of Galois correspondence theorems for groupoid actions in general. In fact, all the Galois correspondences found in the literature require the action to be orthogonal, that is, $A = \bigoplus_{e \in \G_0} A_e$, where $\alpha = (A_g,\alpha_g)_{g \in \G}$ is a partial action of a groupoid $\G$ on a ring $A$ and $\G_0$ is the set of identities of $\G$. This hypothesis covers a broad class of actions, but is still restrictive because it does not cover them all. Every action can be orthogonalized, but the problem is that if we have a Galois extension $A|_{A^{\varepsilon}}$ regarding an orthogonalization $\varepsilon$, it does not imply that the extension $A|_{A^{\alpha}}$ is Galois regarding the original action $\alpha$. In Section 4 we present some examples, as well as we prove two Galois correspondences without the orthogonal hypothesis.

We give now a brief outline of the paper. In the second section we present basic definitions and results about partial groupoid actions and Galois theory. In the third section we prove our first main result, which states that there is a one-to-one correspondence between the wide subgroupoids $\cH$ of the groupoid $\G$ and determined separable, $\alpha$-strong $A^\alpha$-subalgebras $C$ of the algebra $A$, where $\alpha$ is a partial action of $\G$ on $A$ (Theorem \ref{teofund}). In Section 4 we define orthogonalization of a given partial action and we study which properties of Galois theory remain valid under orthogonalization. We also define strongly $\alpha$-Galois extensions and we present in Theorem \ref{teogalnaoort} a Galois correspondence in this case, which is our second main result. Some examples are presented. We end the paper with our third main result, which is a Galois correspondence for global actions without restrictions on the action (Theorem \ref{teofinal}).

By rings, we mean unital rings. We also use the notation $A \triangleleft B$ to indicate that $A$ is a two-sided ideal of the ring $B$.


\section{Preliminaries}

We recall that a \emph{groupoid} \( \G \) is a small category in which every morphism is invertible. Given a morphism \( g \in \G \), we denote its \emph{domain} and \emph{range} by \( d(g) \) and \( r(g) \), respectively. Hence, $d(g) = g^{-1}g$ and $r(g) = gg^{-1}$.  Furthermore, \( \G_0 \) denotes the set of objects of \( \G \). Instead of using $Mor(\G)$ to denote the set of morphisms of $\G$, we identify $\G$ with $Mor(\G)$. We denote $$\G_2 = \{(g,h) \in \G \times \G : \text{the product } gh \text{ is defined} \}.$$ For each $x \in \G_0$, we identify $x$ with the identity morphism $id_x$, thus considering $\G_0\subseteq \G$.

Given $e \in \G_0$, consider $\G(e) = \{ g \in \G : r(g) = d(g) = e \}$. It is easy to verify that $\G(e)$ is a group for all $e \in \G_0$, called \textit{the isotropy group associated to} $e$. A groupoid $\G$ is said to be \emph{connected} if given any $e_1, e_2 \in \G_0$ there exists $g \in \G$ with $d(g) = e_1$ and  $r(g) = e_2$.

Given $e_1,e_2\in\G_0$ we set $\G(e_1,e_2):=\{g\in \mathcal{G}: d(g)=e_1\,\,\text{and}\,\,r(g)=e_2\}$. It is well-known that any groupoid is a disjoint union of connected subgroupoids. Indeed, we define the following equivalence relation in $\G_0$: for all $e_1,e_2 \in \G_0$, 
\begin{align*}
e_1\sim e_2 &\text{ if and only if } \G(e_1,e_2)\neq \emptyset.
\end{align*}

Every equivalence class $\bar{e} \in\G_0/\!\!\sim$ determines a connected subgroupoid $\G_{\bar{e}}$ of $\G$, whose set of identities is $\bar{e}$. The subgroupoid $\G_{\bar{e}}$ is called the {\it connected component of $\G$ associated to $\bar{e}$}. It is clear that $\G=\dot\cup_{\bar{e}\in \G_0/\!\sim}\G_{\bar{e}}$.

Given a nonempty set $S$, we recall that the \emph{coarse groupoid} associated to $S$ is the groupoid $S^2 = S \times S$, where the product $(s,t)(u,v)$ is defined if and only if $s = v$ and in this case $(s,t)(u,v) = (u,t)$. If $\G$ is a connected groupoid,  we have $\G \simeq \G_0^2 \times \G(e)$ (c.f. \cite[Proposition 3.3.6]{lawson1998inverse}), and this isomorphism does not depend on the choice of $e \in \G_0$. Notice that every two finite coarse groupoids with the same amount of identities are isomorphic. Hence we will just denote by $\cA_n$ the coarse groupoid with $n$ identities. Thus, if $\G$ is a finite connected groupoid with cardinality $|\G_0| = n$ and $e \in \G_0$, we have $\G \simeq \cA_n \times \G(e)$.

\subsection{Partial Groupoid Actions}

We begin with the definition of partial groupoid action. 

\begin{defi} \label{defacparcgrpord}
Let $\G$ be a groupoid and $A$ be a ring. We say that $\alpha = (A_g,\alpha_g)_{g \in \G}$ is a \emph{partial action} of $\G$ on $A$ if, for all $g \in \G$ $A_g \triangleleft A_{r(g)} \triangleleft A$, $\alpha_g : A_{g^{-1}} \to A_g$ is an isomorphism of rings, and
\begin{enumerate}
    \item[(P1)] $\alpha_e = \text{Id}_{A_e}$, for all $e \in \G_0$, and $A = \sum_{e \in \G_0} A_e$;
    
    \item[(P2)] $\alpha_{h}^{-1}(A_{g^{-1}} \cap A_h) \subset A_{(gh)^{-1}}$, for all $(g,h) \in \G_2$;
    
    \item[(P3)] $\alpha_g \circ \alpha_h(x) = \alpha_{gh}(x)$, for all $(g,h) \in \G_2, x \in \alpha_h^{-1}(A_{g^{-1}}\cap A_h)$.
\end{enumerate}
\end{defi}

The first thing to note is that condition (P1) is not the same as reference \cite{bagio2012partial}, because we require that $A = \sum_{e \in \G_0} A_e$, and in \cite{bagio2012partial} this condition was not requested. When $\G$ is a group with identity element $1_{\G}$, this condition translates into $A = A_{1_{\G}}$. If we do not require for this, it may happen that a partial groupoid action is not a partial group action, even when the groupoid is a group.

Naturally, a partial action is said to be \emph{global} if $A_g = A_{r(g)}$, for all $g \in \G$.

\begin{prop} \cite[Lemma 1.1]{bagio2012partial} \label{propriedadesacpargrp}
Let $\alpha = (A_g,\alpha_g)_{g \in \G}$ be a partial action of the groupoid $\G$ on the ring $A$. Then it hold:
\begin{enumerate}
    \item[(i)] $\alpha_g^{-1} = \alpha_{g^{-1}}$, for all $g \in \G$.
    
    \item[(ii)] $\alpha_g(A_{g^{-1}} \cap A_h) = A_g \cap A_{gh}$, for all $(g,h) \in \G_2$.
\end{enumerate}
\end{prop}

We say that the partial action $\alpha$ is \emph{preunital} if each $A_e$ is generated by a central idempotent of $A$, denoted by $1_e$, for $e \in \G_0$. We say that a partial action $\alpha$ is \emph{unital} if each $A_g$ is generated by a central idempotent $1_g$, for $g \in \G$.

We say that a partial action $\alpha$ is \emph{orthogonal} if
\begin{enumerate}
    \item[(PO)] $A = \bigoplus_{e \in \G_0} A_e$
\end{enumerate}
holds.

\begin{defi}
Let $\alpha = (A_g,\alpha_g)_{g \in \G}$ be a partial action of a groupoid $\G$ on a ring $A$. A global action $\beta = (B_g,\beta_g)_{g \in \G}$ of $\G$ on a ring $B$ is said to be a \emph{globalization} of $\alpha$ if, for all $e \in \G_0$, there is a monomorphism of rings $\varphi_e : A_e \to B_e$ such that
\begin{enumerate}
    \item[(G1)] $\varphi_e(A_e) \triangleleft B_e$;
    
    \item[(G2)] $\varphi_{r(g)}(A_g) = \varphi_{r(g)}(A_{r(g)}) \cap \beta_g(\varphi_{d(g)}(A_{d(g)}))$, for all $g \in \G$;
    
    \item[(G3)] $\beta_g \circ \varphi_{d(g)}(a) = \varphi_{r(g)}  \circ \alpha_g(a)$, for all $a \in A_{g^{-1}}$, $g \in \G$;
    
    \item[(G4)] $B_g = \sum_{r(h) = r(g)} \beta_h\left ( \varphi_{d(h)}(A_{d(h)}) \right)$, for all $g \in \G$.
\end{enumerate}
\end{defi}

We have a characterization for partial actions that have globalizations.

\begin{theorem} \label{teoglob} \cite[Theorem 2.1]{bagio2012partial}
Let $\alpha = (A_g,\alpha_g)_{g \in \G}$ be a preunital partial action of a groupoid $\G$ on a ring $A$.  Then $\alpha$ has a globalization $\beta$ if and only if $\alpha$ is unital. Moreover, such $\beta$ is unique up to isomorphism.
\end{theorem}

\subsection{Galois Theory}

In this subsection we present some basic concepts regarding Galois theory. From now on, assume that $\G$ is a groupoid, and $\alpha = (A_g,\alpha_g)_{g \in \G}$ is a unital partial action of $\G$ on a ring $A$.  We define
    \begin{align*}
        A^{\alpha} = \{a \in A : \alpha_g(a1_{g^{-1}}) = a1_g, \text{ for all } g \in \G\},
    \end{align*}
    the \emph{subring of invariants of} $A$ \emph{under} $\alpha$.

\begin{defi}
We say that $A$ is a \emph{partial $\alpha$-Galois extension of} $A^\alpha$ if there is $\{x_i,y_i\}_{i=1}^n \subseteq A$ such that
    \begin{align*}
        \sum_{i=1}^n x_i\alpha_g(y_i1_{g^{-1}}) = \sum_{e \in \G_0} \delta_{e,g}1_e,
    \end{align*}
    for all $g \in \G$. In this case, we say that the set $\{x_i,y_i\}_{i=1}^n$ is a set of \emph{partial $\alpha$-Galois coordinates of} $A$ \emph{over} $A^{\alpha}$.
\end{defi}

\begin{defi}
Let $\mathcal{H}$ be a subgroupoid of $\G$. We define the \emph{restriction of $\alpha$ to} $\cH$ as
\begin{align*}
    \alpha|_{\cH} = (A_h,\alpha_h)_{h \in \cH}.
\end{align*}
\end{defi}

\begin{defi}
    Let $C$ be a $A^{\alpha}$-subalgebra of $A$. We define
    \begin{align*}
        \G_C = \{g \in \G : \alpha_g(a1_{g^{-1}}) = a1_g, \text{ for all } a \in C\}.
    \end{align*}
\end{defi}

Clearly, $\G_0 \subseteq \G_C$, for all $C$. However, $\G_C$ is not always a subgroupoid of $\G$. An example of that can be found at \cite[Example 6.3]{dokuchaev2007partial}.

\begin{defi}
    Let $C$ be a $A^\alpha$-subalgebra of a ring $A$. We say that $C$ is:
    \begin{enumerate}
        \item[(i)] $A^{\alpha}$-\emph{separable} if $C$ is a projective ($C \otimes_{A^\alpha} C^{op}$)-module \cite[Ch. III $\S$1]{knus2006theorie};

        \item[(ii)] $\alpha$-\emph{strong} if for all $g,h \in \G$ with $r(g) = r(h)$ and $g^{-1}h \notin \G_C$, and for all nonzero idempotent $e \in A_g \cup A_h$ there is an element $a \in C$ such that $\alpha_g(a1_{g^{-1}})e \neq \alpha_h(a1_{h^{-1}})e$ \cite[Proposition 5.4]{bagio2022galois}.
    \end{enumerate}
\end{defi}

\begin{obs}
    Concerning $C$ being $A^\alpha$-separable, we have the following equivalent definition (c.f. \cite[Théorème 1.4]{knus2006theorie}): $C$ is $A^\alpha$-separable if there is an idempotent $e = \sum_{i=1}^n x_i \otimes y_i \in C \otimes_{A^\alpha} C$ such that $\displaystyle \sum_{i=1}^n x_iy_i = 1_C$ and
 $e(1_C \otimes x - x \otimes 1_C) = 0$, for all $x \in S$. We say that $e$ is an \emph{idempotent of separability} of the extension $C|_{A^\alpha}$.
\end{obs}

\section{Correspondence Theorem}

 This section is intended to prove a Galois correspondence for partial groupoid actions. From now on, assume that the groupoid $\G$ is finite.
 
 Take $e \in \G_0$. Consider $\G(-,e) = \{ g \in \G : r(g) = e \}$. Suppose that $\G(-,e) = \{e = g_{1,e}, g_{2,e} \ldots, g_{n_e,e}\}$, where $n_e$ is the cardinality $|\G(-,e)|$ of $\G(-,e)$. Consider the following assertion: \begin{align}\label{superfluous}
      \text{For each } g \in \G, g_{r(g),i}^{-1}gg_{d(g),i} \in \G_0 \text{ implies that } g \in \G_0, \text{ for all } 1 \leq i \leq n_{r(g)} = n_{d(g)}.
 \end{align}

 In \cite[Theorem 5.1]{bagio2012partial}, it was proved that if \eqref{superfluous} holds, then $A$ being a partial $\alpha$-Galois extension of $A^{\alpha}$ implies that $B$ is a $\beta$-Galois extension of $B^\beta$, where $\beta$ is the globalization of $\alpha$, a global action of $\G$ on a ring $B$.  We shall prove that this hypothesis is superfluous. So we start this section by improving \cite[Theorem 5.1]{bagio2012partial}.

\begin{lemma} \label{lemagaloisparcialsseglobal}
    Assume that $\alpha$ is orthogonal. Consider a global action $\beta$ of $\G$ on a ring $B$, and assume that $\beta$ is a globalization of $\alpha$. The following statements are equivalent:
    \begin{enumerate}
        \item[(i)] $A$ is a partial $\alpha$-Galois extension of $A^{\alpha}$.

        \item[(ii)] $B$ is a $\beta$-Galois extension of $B^\beta$.
    \end{enumerate}
\end{lemma}
\begin{proof}
    The implication (ii) $\Rightarrow$ (i) was proved in \cite[Theorem 5.1]{bagio2012partial}.

    (i) $\Rightarrow$ (ii): Assume that $\G$ is connected. For the general case, it is enough to obtain a convenient enumeration for each connected component of $\G$.

    Since $\G$ is finite and connected, $\G \simeq \mathcal{A}_n \times G$, where $\mathcal{A}_n$ is the coarse groupoid with $n$ identities and $G$ is a finite group with $m$ elements \cite[Proposition 3.3.6]{lawson1998inverse}. Observe that $n_e = |\G(-,e)| = nm$, for all $e \in \G_0$. Now enumerate $G$ as the following: $G = \{1_G = g_0, \ldots, g_{m-1}\}$.  Let $\mathbb{S}_n$ be the group of permutation of $n$ elements. Consider $\sigma \in \mathbb{S}_n$ given by $\sigma(i) = i-1$, for all $1 < i \leq n$, and $\sigma(1) = n$. Hence, given $1 \leq j \leq mn$, there are unique $0 \leq q \leq m-1$ and $0 \leq r \leq n-1$ such that $j = qn + r + 1$. We define, then, $g_{e_i,j} = (e_{\sigma^r(i)},e_i,g_q)$.

    In this enumeration, $g_{e_i,1} = (e_{\sigma^0(i)},e_i,g_0) = (e_i,e_i,1_G) = e_i$. Given $\ell = (e_i,e_j,g_k) \in \G$ and given $1 \leq u = qn + r \leq mn$, we have:
    \begin{align*}
        g_{r(\ell),u}^{-1}\ell g_{d(\ell),u} & = g_{e_j,u}^{-1}\ell g_{e_i,u} = (e_{\sigma^r(j)},e_j,g_q)^{-1}(e_i,e_j,g_k)(e_{\sigma^r(i)},e_i,g_q) \\
        & = (e_j,e_{\sigma^r(j)},g_q^{-1})(e_i,e_j,g_k)(e_{\sigma^r(i)},e_i,g_q) = (e_{\sigma^r(i)},e_{\sigma^r(j)},g_q^{-1}g_kg_q).
    \end{align*}

    Therefore
    \begin{align*}
        g_{r(\ell),u}^{-1}\ell g_{d(\ell),u} \in \G_0 & \Leftrightarrow \sigma^r(i) = \sigma^r(j) \text{ and } g_q^{-1}g_kg_q = 1_G \\
        & \Leftrightarrow i = j \text{ and } g_k = g_qg_q^{-1} = 1_G \Leftrightarrow \ell \in \G_0,
    \end{align*}
    which ends the proof.
\end{proof}

Throughout this section, assume that $\alpha$ is orthogonal and that the ring $A$ is commutative. We fix the notation $\beta = (B_g,\beta_g)_{g\in\G}$ to denote a global action of $\G$ on a ring $B$ that is a globalization of $\alpha$. According to \cite[Theorem 2.1]{bagio2012partial}, it follows that $B = \prod_{e \in \G_0} B_e$, so that $\beta$ is an orthogonal action. To shorten the notation, we will assume $A_e \subseteq B_e$, for all $e \in \G_0$, instead of $\varphi_e(A_e) \subseteq B_e$. We will also denote by $1_g$ the identity of $A_g$ and by $1_g'$ the identity of $B_g$, for all $g \in \G$.

Now we shall prove the first part of the Galois correspondence.

\begin{theorem} \label{teogalpargrp1}
    Let $A$ be a partial $\alpha$-Galois extension of $A^\alpha$ and $\cH$ a wide subgroupoid of $\G$. Then $\alpha|_{\cH}$ is a partial action of $\cH$ on $A$ and $A$ is a partial $\alpha|_{\cH}$-Galois extension of $C = A^{\alpha|_{\cH}}$. Furthermore, $C$ is $A^{\alpha}$-separable, $\alpha$-strong and $\G_C = \cH$.
\end{theorem}
\begin{proof}
    It is easy to see that $\alpha|_{\cH}$ is an orthogonal partial action of $\cH$ on $A$ and that $A$ is a partial $\alpha|_{\cH}$-Galois extension of $C$ with the same partial Galois coordinates of $A$ over $A^\alpha$.

    Consider the globalization $\beta$ of $\alpha$. Define $B_g' = \sum_{\substack{h \in \cH \\ r(h) = r(g)}} \beta_h(A_{d(h)})$ and $\beta_g' = \beta_g|_{B_{g^{-1}}'}$. Then $\beta' = (B_g',\beta_g')_{g\in\cH}$ is a global action  of $\cH$ on $B' = \prod_{e \in \cH_0} B_e'$. In fact, 
    \begin{align*}
        \beta_h'(B_{h^{-1}}') & = \beta_h' \left ( \sum_{\substack{k \in \cH \\ r(k) = d(h)}} \beta_k(A_{d(k)}) \right )  = \sum_{\substack{k \in \cH \\ r(k) = d(h)}} \beta_h(\beta_k(A_{d(k)})) \\
        & = \sum_{\substack{k \in \cH \\ r(k) = d(h)}} \beta_{hk}(A_{d(k)}) = \sum_{\substack{\ell \in \cH \\ r(\ell) = r(h)}} \beta_{\ell}(A_{d(\ell)}) = B_h'.
    \end{align*}

    Since $\beta_h$ is an isomorphism for all $h \in \cH$, it follows that $\beta_h'$ is also an isomorphism. Clearly $\beta_e' = \beta_e|_{B_e'} = \text{Id}_{B_e}|_{B_e'} = \text{Id}_{B_e'}$, for all $e \in \cH_0$, and given $(h,k) \in \cH_2$, $\beta_h' \circ \beta_k' = \beta_h|_{B_{d(h)}'} \circ \beta_k|_{B_{d(k)}'} =  \beta_h|_{B_{r(k)}'} \circ \beta_k|_{B_{d(k)}'} = \beta_h \circ \beta_k|_{B_{d(k)}'} = \beta_{hk}|_{B_{d(hk)}'} = \beta_{hk}'$.
    
    Moreover, $\beta'$ is the globalization of $\alpha|_{\cH}$. Indeed, $A_e \triangleleft B_e'$ by definition, for all $e \in \cH_0$. Furthermore, $A_h = A_{r(h)} \cap \beta_h(A_{d(h)})$, for all $h \in \cH$. In particular, since $A_{d(h)} \subseteq B_{d(h)}'$, it follows that $\beta_h(A_{d(h)}) = \beta_h'(A_{d(h)})$, from where $A_h = A_{r(h)} \cap \beta_h'(A_{d(h)})$, for all $h \in \cH$. By the same reason it follows that $\beta_h'(a) = \alpha_h(a)$, for all $a \in A_{h^{-1}}, h \in \cH$. Finally, the condition (G4) is precisely the definition of $B_h$, for all $h \in \cH$.

    The action $\beta$ of $\G$ on $B$ induces a partial action $\gamma$ of $\G$ on $B'$ via restriction, and $\beta$ is the globalization of $\gamma$. Hence, the proof of \cite[Theorem 4.6]{bagio2012partial} implies that $B^\beta1_{B'} = (B')^{\gamma}$ and $B^{\beta}1_A = A^\alpha$. Therefore, $(B')^\gamma1_A = B^\beta1_{B'}1_A = B^\beta1_A = A^{\alpha}$. Furthermore, $(B')^{\beta'}1_A = A^{\alpha|_{\cH}}$.

    On the other hand, $B = B1_{B'} \oplus B(1_B - 1_{B'}) = B' \oplus B(1_B - 1_{B'})$. We have that $\beta_h(B(1_B - 1_{B'})1_{h^{-1}}') \subseteq B(1_B - 1_{B'})$, for all $h \in \cH$, as
    \begin{align*}
        \beta_h(b(1_B - 1_B')1_{h^{-1}}') & = \beta_h(b1_{h^{-1}}') - \beta_h(b1_{B'}1_{h^{-1}}') = \beta_h(b1_{h^{-1}}') - \beta_h(b1_{h^{-1}}')\beta_h(1_{B'}1_{h^{-1}}') \\
        & = \beta_h(b1_{h^{-1}}') - \beta_h(b1_{h^{-1}}')\beta_h'(1_{h^{-1}}'') = \beta_h(b1_{h^{-1}}') - \beta_h(b1_{h^{-1}}')1_h'' \\
        & = \beta_h(b1_{h^{-1}}')1_B - \beta_h(b1_{h^{-1}}')1_{B'} = \beta_h(b1_{h^{-1}}')(1_B - 1_{B'}).
    \end{align*}
    
    Thus, $B^{\beta|_{\cH}} = (B')^{\beta|_{\cH}} \oplus (B(1_B - 1_{B'}))^{\beta|_{\cH}}$ and $B^{\beta|_{\cH}}1_{B'} = (B')^{\beta|_{\cH}}1_{B'} = (B')^{\beta|_{\cH}} = (B')^{\beta'}$ and hence $C = A^{\alpha|_{\cH}} = (B')^{\beta'}1_A = B^{\beta|_{\cH}}1_{B'}1_A = B^{\beta|_{\cH}}1_A$.

    By Lemma \ref{lemagaloisparcialsseglobal}, we have that $B$ is a $\beta$-Galois extension of $B^\beta$. So, by \cite[Theorem 4.1]{paques2018galois}, the $B^\beta$-subalgebra $B^{\beta|_{\cH}}$ is separable, $\beta$-strong and $\cH = \G_{B^{\beta|_{\cH}}}' := \{g \in \G : \beta_g(b1_{g^{-1}}') = b1_g', \text{ for all } b \in B^{\beta|_{\cH}}\}$. If $e' \in B^{\beta|_{\cH}} \otimes_{B^\beta} B^{\beta|_{\cH}}$ is the idempotent of separability of $B^{\beta|_{\cH}}$ over $B^\beta$, then $e = e'(1_A \otimes 1_A) \in B^{\beta|_{\cH}}1_A \otimes_{B^\beta1_A} B^{\beta|_{\cH}}1_A = A^{\alpha|_{\cH}} \otimes_{A^\alpha} A^{\alpha|_{\cH}} = C \otimes_{A^\alpha} C$ is the idempotent of separability of $C$ over $A^\alpha$.

    Now we prove that $C$ is $\alpha$-strong. Let $g,h \in \G$ be such that $(g^{-1},h) \in \G_2$, $g^{-1}h \notin \G_C$ and $e \in A_g \cup A_h$ a nonzero idempotent. Since $B^{\beta|_{\cH}}$ is $\beta$-strong and $g^{-1}h \notin \G_C \supseteq \cH = \G_{B^{\beta|_{\cH}}}'$, if $e1_g1_h \neq 0$, there is $b \in B^{\beta|_{\cH}}$ such that $\beta_g(b1_{g^{-1}}')e1_g1_h \neq \beta_h(b1_{h^{-1}}')e1_g1_h$. Hence $b1_A \in B^{\beta|_{\cH}}1_A = C$ and $\alpha_g(b1_A1_{g^{-1}})e1_g1_h = \alpha_g(b1_{g^{-1}})e1_g1_h  = \beta_g(b1_{g^{-1}}')e1_g1_h \neq \beta_h(b1_{h^{-1}}')e1_g1_h = \alpha_h(b1_{h^{-1}})e1_g1_h = \alpha_h(b1_A1_h^{-1})e1_g1_h$, from where we obtain $\alpha_g(b1_A1_{g^{-1}})e \neq \alpha_h(b1_A1_{h^{-1}})e$. If $e1_g1_h = 0$ and $e \in A_g$, then $\alpha_g(1_A1_{g^{-1}})e = 1_ge = e \neq 0 = e1_g1_h = e1_h = \alpha_h(1_A1_{h^{-1}})e$. The case $e \in A_h$ is similar.

    Finally, it only remains to show that $\G_C \subseteq \cH$, since we already have that $\cH \subseteq \G_C$. Assume, by contradiction, that there is $g \in \G_C \setminus \cH$. Since $B^{\beta|_{\cH}}$ is $\beta$-strong, there is $b \in B^{\beta|_{\cH}}$ such that $\beta_g(b1_{g^{-1}}')1_g \neq \beta_{r(g)}(b1_{r(g)}')1_g = b1_g$. Hence, $\alpha_g(b1_A1_{g^{-1}}') = \beta_g(b1_{g^{-1}}')1_g \neq b1_g = b1_A1_g$, and it is a contradiction because $g \in \G_C$ and $b1_A \in B^{\beta|_{\cH}}1_A = C$. Thus $\cH$ must coincide with $\G_C$, ending the proof.
\end{proof}

This result proves the first part of the Galois Correspondence theorem. We also observe that Theorem \ref{teogalpargrp1} shows that if $A|_{A^{\alpha}}$ is a partial $\alpha$-Galois extension, then $A$ is $A^{\alpha}$-separable and $\alpha$-strong. We will prove in the next result that the converse also holds.

\begin{lemma}
     $A|_{A^{\alpha}}$ is a partial $\alpha$-Galois extension if and only if $A$ is $A^{\alpha}$-separable and $\alpha$-strong.
\end{lemma}
\begin{proof}
    By the discussion above, it only remains to prove that if $A$ is $A^{\alpha}$-separable and $\alpha$-strong, then $A$ is partial $\alpha$-Galois over $A^{\alpha}$.

    For that, let $e = \sum_{i=1}^n x_i \otimes y_i \in A \otimes_{A^{\alpha}} A$ be the idempotent of separability of $A$. That is, $\mu(e) = \sum_{i=1}^n x_iy_i = 1_A$ and $(a \otimes 1_A - 1_A \otimes a)e = 0$, for all $a \in A$, where $\mu : A \otimes_{A^{\alpha}} A$ is the multiplication map. For all $g \in \G$, define $e_g = m((\text{Id}_A \otimes \alpha_g(-1_{g^{-1}}))(e)) \in A_g$. We have that $\text{Id}_A \otimes \alpha_g(-1_{g^{-1}})$ is an endomorphism of $A \otimes_{A^\alpha} A$. Besides that, $\mu$ is a ring homomorphism, since $A$ is commutative. Hence, for all $g \in \G$ we have
    \begin{align*}
        e_g^2 & = \mu((\text{Id}_A \otimes \alpha_g(-1_{g^{-1}}))(e)) \cdot \mu((\text{Id}_A \otimes \alpha_g(-1_{g^{-1}}))(e)) \\
        & = \mu((\text{Id}_A \otimes \alpha_g(-1_{g^{-1}}))(e) \cdot (\text{Id}_A \otimes \alpha_g(-1_{g^{-1}}))(e)) \\
        & = \mu((\text{Id}_A \otimes \alpha_g(-1_{g^{-1}}))(e^2)) = \mu((\text{Id}_A \otimes \alpha_g(-1_{g^{-1}}))(e)) = e_g.
    \end{align*}

   That is, $e_g$ is an idempotent in $A_g$. On the one hand, for all $a \in A$, we get
    \begin{align*}
        \alpha_{r(g)}(a1_{r(g)})e_g & = ae_g = a\mu((\text{Id}_A \otimes \alpha_g(-1_{g^{-1}}))(e)) = a\mu((\text{Id}_A \otimes \alpha_g(-1_{g^{-1}}))(e))1_g \\
        & = (a \otimes 1_g) \cdot \mu((\text{Id}_A \otimes \alpha_g(-1_{g^{-1}}))(e)) = \mu((a \otimes 1_g) \cdot (\text{Id}_A \otimes \alpha_g(-1_{g^{-1}}))(e)) \\
        & = \mu((\text{Id}_A \otimes \alpha_g(-1_{g^{-1}}))(a \otimes 1_A) \cdot (\text{Id}_A \otimes \alpha_g(-1_{g^{-1}}))(e)) \\
        & = \mu((\text{Id}_A \otimes \alpha_g(-1_{g^{-1}}))((a \otimes 1_A)e)) = \mu((\text{Id}_A \otimes \alpha_g(-1_{g^{-1}}))((1_A \otimes a)e)) \\
        & = \mu((\text{Id}_A \otimes \alpha_g(-1_{g^{-1}}))(1_A \otimes a) \cdot (\text{Id}_A \otimes \alpha_g(-1_{g^{-1}}))(e)) \\
        & = \mu((1_A \otimes \alpha_g(a1_{g^{-1}})) \cdot (\text{Id}_A \otimes \alpha_g(-1_{g^{-1}}))(e)) \\
        & = \mu((1_A \otimes \alpha_g(a1_{g^{-1}}))) \cdot \mu((\text{Id}_A \otimes \alpha_g(-1_{g^{-1}}))(e)) \\
        & = \alpha_g(a1_{g^{-1}})\mu((\text{Id}_A \otimes \alpha_g(-1_{g^{-1}}))(e)) = \alpha_g(a1_{g^{-1}})e_g.
    \end{align*}

    On the other hand, since $A$ is $\alpha$-strong, we have that if $g \notin \G_0$, then we should have $e_g = 0$. But
    \begin{align*}
        e_g & = \mu((\text{Id}_A \otimes \alpha_g(-1_{g^{-1}}))(e)) = \sum_{i=1}^n x_i\alpha_g(y_i1_{g^{-1}}).
    \end{align*}

    If $g \in \G_0$, then $g = r(g)$ and
    \begin{align*}
        e_g & = \mu((\text{Id}_A \otimes \alpha_{r(g)}(-1_{r(g)}))(e)) = \sum_{i=1}^n x_i\alpha_{r(g)}(y_i1_{r(g)}) \\
        & = \sum_{i=1}^n x_iy_i1_{r(g)}  = 1_A1_{r(g)} = 1_{r(g)}.
    \end{align*}

    Therefore $\{x_i,y_i\}_{i=1}^n$ is a partial $\alpha$-Galois coordinate system of $A$ over $A^{\alpha}$.
\end{proof}

Before we prove the converse of Theorem \ref{teogalpargrp1}, we need a technical lemma.

\begin{lemma} \label{lemacoord}
    Let $A$ be a partial $\alpha$-Galois extension of $A^\alpha$ and $C$ a separable, $\alpha$-strong $A^\alpha$-subalgebra of $A$. Then there exist $x_i,y_i \in C$, $1 \leq i \leq n$, such that $\sum_{i = 1}^n x_iy_i = 1_A$ and $\sum_{i = 1}^n x_i \alpha_g(y_i1_{g^{-1}}) = 0$, for all $g \in \G \setminus \G_C$.
\end{lemma}
\begin{proof}
    Let $e = \sum_{i=1}^n x_i \otimes y_i \in C \otimes_{A^{\alpha}} C$ be the idempotent of separability of $C$ over $A^\alpha$. By definition we have that $\sum_{i=1}^n x_iy_i = 1_A$. Take $g \in \G$ and define $e_{g} = \sum_{i=1}^n x_i\alpha_g(y_i1_{g^{-1}}) \in A_g$, which is an idempotent since $e_g = \mu(\theta_g(e))$, where $\theta_g : C \otimes C \to C \otimes A_g$ is the $C$-algebra homomorphism given by $\theta_g \left ( \sum_{i=1}^m a_i \otimes b_i \right) = \sum_{i=1}^m a_i\alpha_g(b_i1_{g^{-1}})$ and $\mu : C \otimes C \to C$ is the product map, which is a homomorphism since $C$ is commutative.

    Observe now that $ae_g = \alpha_g(a1_g^{-1})e_g$, for all $a \in C$. In fact,
    \begin{align*}
        ae_g & = ae_g1_g = a\mu(\theta_g(e))1_g = a \mu((\text{Id}_C \otimes \alpha_g(-1_{g^{-1}}))(e))1_g \\
        & = a \mu((\text{Id}_C \otimes \alpha_g(-1_{g^{-1}}))(e)) 1_g = (a \otimes 1_g) \cdot \mu((\text{Id}_C \otimes \alpha_g(-1_{g^{-1}}))(e)) \\
        & = \mu((a \otimes 1_g) \cdot (\text{Id}_C \otimes \alpha_g(-1_{g^{-1}}))(e)) \\
        & = \mu((\text{Id}_C \otimes \alpha_g(-1_{g^{-1}}))(a \otimes 1_A) \cdot (\text{Id}_C \otimes \alpha_g(-1_{g^{-1}}))(e)) \\
        & = \mu((\text{Id}_C \otimes \alpha_g(-1_{g^{-1}}))((a \otimes 1_A) \cdot e)) \\
        & = \mu((\text{Id}_C \otimes \alpha_g(-1_{g^{-1}}))((1_A \otimes a) \cdot e)) \\
        & = \mu((\text{Id}_C \otimes \alpha_g(-1_{g^{-1}}))(1_A \otimes a) \cdot (\text{Id}_C \otimes \alpha_g(-1_{g^{-1}}))(e)) \\
        & = \mu((1_A \otimes \alpha_g(a1_{g^{-1}}) \cdot (\text{Id}_C \otimes \alpha_g(-1_{g^{-1}}))(e)) \\
        & = (1_A \otimes \alpha_g(a1_{g^{-1}}) \cdot \mu((\text{Id}_C \otimes \alpha_g(-1_{g^{-1}}))(e)) \\
        & = \alpha_g(a1_{g^{-1}})\mu((\text{Id}_C \otimes \alpha_g(-1_{g^{-1}}))(e)) = \alpha_g(a1_{g^{-1}})e_g.
        \end{align*}

    Since $C$ is $\alpha$-strong, it follows that $e_g = 0$ unless $g \in \G_C$.
\end{proof}

\begin{theorem} \label{teogalpargrp2}
    Let $A$ be a partial $\alpha$-Galois extension of $A^\alpha$ and $C$ a separable, $\alpha$-strong $A^\alpha$-subalgebra of $A$ such that $\G_C$ is a wide subgroupoid of $\G$. Then $A^{\alpha|_{\G_C}} = C$.
\end{theorem}
\begin{proof}
    We always have that $C \subseteq A^{\alpha|_{\G_C}}$. Our goal is to prove the reverse inclusion. Let $\beta$ be the globalization of $\alpha$. Consider the subalgebra
    \begin{align*}
        B' = \prod_{e \in \G_0} \left ( \sum_{\substack{g \in \G_C \\ r(g) = e}} \beta_g(A_{d(g)}) \right ),
    \end{align*}
    where $\G_C$ acts via $\beta'$ as in the proof of Theorem \ref{teogalpargrp1} and $\beta'$ is the globalization of the orthogonal partial action $\alpha|_{\G_C}$ of $\G_C$ on $A$.

    By \cite[Theorem 4.6]{bagio2012partial}, there is a ring homomorphism $\psi_{\G_C} : A \to B'$ such that $\psi_{\G_C}|_{A^{\alpha|_{\G_C}}} : A^{\alpha|_{\G_C}} \to (B')^{\beta'}$ is a ring homomorphism and $\psi_{\G_C}(a)1_A = a$, for all $a \in A^{\alpha|_{\G_C}}$. Consider $C' = \psi_{\G_C}(C)$.

    \noindent \textbf{Claim 1:} There exist $x_i', y_i' \in C'$, $1 \leq i \leq n$, such that $\sum_{i=1}^n x_i'y_i' = 1_{B'}$ and such that  $\sum_{i=1}^n x_i'\beta_g'(y_i'1_{g^{-1}}') = 0$, for all $g \in \G \setminus \G_C$.

    In fact, by Lemma \ref{lemacoord}, there exist $x_i,y_i \in C$, $i \leq i \leq n$, such that $\sum_{i=1}^n x_iy_i = 1_A$ and $\sum_{i=1}^n x_i\alpha_{g}(y_i1_{g^{-1}}) = 0$, for all $g \in \G \setminus \G_C$. Since $x_i, y_i \in C \subseteq A$, given $g \in \G \setminus \G_C$, we have that
    \begin{align*}
        \sum_{i=1}^n x_i\beta_g(y_i1_{g^{-1}}') & = \sum_{i=1}^n x_i\beta_g(y_i1_A1_{g^{-1}}') = \sum_{i=1}^n x_i\beta_g(y_i1_{g^{-1}}) = \sum_{i=1}^n x_i\alpha_g(y_i1_{g^{-1}}) = 0.
    \end{align*}

    Take $x_i' = \psi_{\G_C}(x_i)$ and $y_i' = \psi_{\G_C}(y_i)$. Then $\sum_{i=1}^n x_i'y_i' = \psi_{\G_C} \left ( \sum_{i=1}^n x_iy_i \right ) = \psi_{\G_C}(1_A) = 1_{B'}$. Consider now $e \in \G_0$ and $\{h_{1,e} = e, h_{2,e}, \ldots, h_{m_e,e}\} = \G_C(-,e) = \{g \in \G_C : r(g) = e\}$, where $m_e > 0$. By \cite[pg. 15]{bagio2012partial}, $1_e' = v_{1,e} + \cdots + v_{m_e,e}$, where $v_{1,e} = 1_e$ and $v_{j,e} = (1_{e}' - 1_e)(1_e' - \beta_{h_{2,e}}(1_{d(h_{2,e})})) \cdots (1_e' - \beta_{h_{j-1,e}}(1_{d(h_{j-1,e})}))\beta_{h_{j,e}}(1_{d(h_{j,e})})$. 
    
    Defining $\psi_e : A \to B'$ by $\psi_e(a) = \sum_{j=1}^{m_e} \beta_{h_{j,e}}(a_{d(h_{j,e})}1_{h_{j,e}^{-1}}') v_{j,e}$, we have that $\psi_{\G_C} = \sum_{e \in \G_0} \psi_e$. Hence, for all $g \in \G \setminus \G_C$, we have that
    
    \begin{align*}
        & \sum_{i=1}^n x_i'\beta_g(y_i'1_{g^{-1}}') \\
        & = \sum_{i=1}^n \sum_{e \in \G_0} \sum_{1 \leq j,k \leq m_e} \beta_{h_{j,e}}((x_i)_{d(h_{j,e})}1_{h_{j,e}^{-1}}')v_{j,e}\beta_g(\beta_{h_{k,e}}((y_i)_{d(h_{k,e})}1_{h_{k,e}^{-1}}')v_{k,e}1_{g^{-1}}') \\
        & = \sum_{i=1}^n \sum_{1 \leq j,k \leq m_{r(g)}} \beta_{h_{j,r(g)}}((x_i)_{d(h_{j,r(g)})}1_{h_{j,r(g)}^{-1}}')v_{j,r(g)}\beta_g(\beta_{h_{k,d(g)}}((y_i)_{d(h_{k,d(g)})}1_{h_{k,d(g)}^{-1}}')v_{k,d(g)}) \\
        & = \sum_{1 \leq j,k \leq m_{r(g)}} v_{j,r(g)}\beta_g(v_{k,d(g)}) \beta_{h_{j,r(g)}} \left ( \sum_{i=1}^n (x_i)_{d(h_{j,r(g)})}\beta_{h_{j,r(g)}^{-1}gh_{k,d(g)}}((y_i)_{d(h_{k,d(g)})}1_{h_{k,d(g)}^{-1}}') \right ) \\
        & = \sum_{1 \leq j,k \leq m_{r(g)}} v_{j,r(g)}\beta_g(v_{k,d(g)}) \beta_{h_{j,r(g)}} \left ( \sum_{i=1}^n x_i\beta_{h_{j,r(g)}^{-1}gh_{k,d(g)}}(y_i 1_{h_{k,d(g)}^{-1}}') \right ) = 0,
    \end{align*}
    because if $h_{j,r(g)}^{-1}gh_{k,d(g)} \in \G_C$, then there is $h \in \G_C$ such that $h_{j,r(g)}^{-1}gh_{k,d(g)} = h$. But then $g = h_{j,r(g)}hh_{k,d(g)}^{-1}$, from where it follows that $g \in \G_C$, a contradiction. This completes the proof of Claim 1.

    Now, since $\G_C \subseteq \G_{C'}$ and the elements $x_i',y_i' \in C'$, we have, in particular, that $\G_{C'} = \G_C$.

    Consider again the restriction $\gamma$ of $\beta$ to $B'$. It follows from \cite[Theorem 4.6]{bagio2012partial} that there is a ring isomorphism $B^\beta \to (B')^\gamma$ given by $x \mapsto x1_{B'}$. Moreover, $x \mapsto x1_A$ is an isomorphism of $B^\beta$ on $A^\alpha$. Hence, there is an isomorphism $(B')^\gamma \to A^\alpha$ given by $x1_{B'} \mapsto x1_A$, whose inverse is $\psi_{\G_C}|_{A^{\alpha}}$. Therefore $\psi_{\G_C}(A^{\alpha}) = (B')^{\gamma}$ and hence $C' = \psi_{\G_C}(C)$ is separable over $(B')^\gamma$.

    Notice that $B^{\beta|_{\G_C}} = B^{\beta|_{\G_C}}1_{B'} \oplus B^{\beta|_{\G_C}}(1_B - 1_{B'}) = (B1_{B'})^{\beta|_{\G_C}} \oplus B^{\beta|_{\G_C}}(1_B - 1_{B'}) = (B')^{\beta|_{\G_C}} \oplus B^{\beta|_{\G_C}}(1_B - 1_{B'}) = (B')^{\beta'} \oplus B^{\beta|_{\G_C}}(1_B - 1_{B'})$. In particular, $B^{\beta|_{\G_C}}1_A = (B')^{\beta'}1_A = A^{\alpha|_{\G_C}}$. Consider the subalgebra $\overline{C} = C' \oplus B^{\beta|_{\G_C}}(1_{B} - 1_{B'})$ of $B^{\beta|_{\G_c}}$.

    \noindent \textbf{Claim 2:} $\overline{C}$ is $B^\beta$-separable and $\beta$-strong.

    Indeed, $B^{\beta|_{\G_C}}$ is separable over $B^\beta$ by \cite[Theorem 4.1(ii)]{paques2018galois}, from where it follows that $B^{\beta|_{\G_C}}(1_B - 1_{B'})$ is separable over $B^{\beta}(1_B - 1_{B'})$. Since $C'$ is $(B')^\gamma$-separable and $(B')^\gamma = B^\beta1_{B'}$, it follows that $\overline{C}$ is separable over $(B')^\gamma \oplus B^\beta(1_B - 1_{B'}) = B^\beta$.

    To prove that $\overline{C}$ is $\beta$-strong, assume that $g \in \G \setminus \G_C$, that $e \in B_g$ is an idempotent and that $\beta_g((x+y)1_{g^{-1}}')e = (x+y)e$, for all $x \in C'$ and $y \in B^{\beta|_{\G_C}}(1_B - 1_{B'})$. 

    Write $e = e_1 + e_2$, with $e_1 = e1_{B'}$ and $e_2 = e(1_{B} - 1_B')$. Observe that $e_1, e_2 \in B_g$. Then multiplying $\beta_{g}((x+y)1_{g^{-1}}')e = (x+y)e$ by $1_{B'}$, we obtain $\beta_g((x+y)1_{g^{-1}})e_1 = (x+y)e_1 = xe_1$, for all $x \in C'$ and $y \in B^{\beta|_{\G_C}}(1_B - 1_{B'})$. In particular, when $y = 0$ we have that $\beta_g(x1_{g^{-1}}')e_1 = xe_1$, for all $x \in C'$.

    By Claim 1, there are $x_i',y_i' \in C'$, $1 \leq i \leq n$, such that $\sum_{i=1}^n x_i'y_i' = 1_{B'}$ and such that $\sum_{i = 1}^n x_i'\beta_g(y_i'1_{g^{-1}}') = 0$, for all $g \in \G \setminus \G_C$. Thus $0 = \sum_{i = 1}^n x_i'\beta_g(y_i'1_{g^{-1}}')e_1 = \sum_{i=1}^n x_i'y_i'e_1 = 1_{B'}e_1 = e_1$. So we obtain $e = e_2$ and $\beta_g((x+y)1_{g^{-1}}')e_2 = (x+y)e_2 = ye_2$, for all $x \in C'$ and $y \in B^{\beta|_{\G_C}}(1_B - 1_{B'})$. Considering $x = 0$, we obtain $\beta_g(y1_{g^{-1}}')e_2 = ye_2$, for all $y \in B^{\beta|_{\G_C}}(1_B - 1_{B'})$. Since $B^{\beta|_{\G_C}}$ is $\beta$-strong and separable over $B^\beta$, it follows that there are $u_j, v_j \in B^{\beta|_{\G_C}}$, $1 \leq j \leq m$, such that $\sum_{j=1}^m u_jv_j = 1_{B}$ and $\sum_{j=1}^m u_j\beta_g(v_j1_{g^{-1}}') = 0$, for all $g \in \G \setminus \G_C$. Thus $0 = \sum_{j=1}^m u_j\beta_g(v_j(1_B - 1_{B'})1_{g^{-1}}')e_2 = \sum_{j=1}^m u_jv_j\beta_g((1_B - 1_{B'})1_{g^{-1}}')e_2 = 1_{B}(1_g' - 1_g'')e_2 = (1_B - 1_{B'})e_2 = e_2$, from where it follows that $e = 0$, that is, Claim 2 holds.

    Now, by Claim 2 and \cite[Theorem 4.5]{paques2018galois}, $\overline{C} = B^{\beta|_{\G_{\overline{C}}'}}$, where $\G_{\overline{C}}' = \{ g \in \G : \beta_h(x1_{h^{-1}}') = x1_h', \text{ for all } x \in \overline{C}\}$. Besides that, by definition of $\overline{C}$, $g \in \G$ is in $\G_{\overline{C}} = \G_{\overline{C}}'$ if and only if $g \in \G_{C'} = \G_C$, from where $\G_C = \G_{\overline{C}}$. It follows that $\overline{C} = B^{\beta|_{\G_C}}$ and $A^{\alpha|_{\G_C}} = (B')^{\beta'}1_A = B^{\beta|_{\G_C}}1_{B'}1_A = B^{\beta|_{\G_C}}1_A = \overline{C} 1_A = C' 1_A = C$.
    \end{proof}

    Now we can state the main result of this section. This theorem generalizes the Galois Correspondence for orthogonal groupoid actions \cite[Theorem 4.6]{paques2018galois}, for orthogonal group-type partial groupoid actions \cite[Theorem 5.7]{bagio2022galois} and for partial group actions \cite[Theorem 5.1]{dokuchaev2007partial}.

    \begin{theorem}[Galois Correspondence for Orthogonal Partial Groupoid Actions]\label{teofund}
        Let $\G$ be a finite groupoid acting partially on a ring $A$ via unital and orthogonal partial action $\alpha = (A_g,\alpha_g)_{g \in \G}$ such that $A_g \neq 0$, for all $g \in \G$. Assume that $A$ is a partial $\alpha$-Galois extension of $A^\alpha$. Then there is a one-to-one correspondence between the wide subgroupoids $\cH$ of $\G$ and the separable, $\alpha$-strong $A^\alpha$-subalgebras $C$ of $A$ such that $\G_C$ is a subgroupoid of $\G$, given by $\cH \mapsto A^{\alpha|_{\cH}}$ with inverse $C \mapsto \G_C$.
    \end{theorem}
    \begin{proof}
        It follows by Theorems \ref{teogalpargrp1} and \ref{teogalpargrp2}.
    \end{proof}

\section{General Actions}

Until now, Galois theory for groupoid actions has required the orthogonality of the action and the reason is to assure the invariance of the trace map. This hypothesis, even though it covers a large class of actions, is somewhat restrictive. So the goal of this section is to develop a Galois theory for nonorthogonal partial actions of groupoids. 

\subsection{Orthogonalization}
\begin{defi}
Two partial actions $\alpha = (A_g,\alpha_g)_{g \in \G}$ of $\G$ on a ring $A$ and $\varepsilon = (E_g,\varepsilon_g)_{g \in \G}$ of $\G$ on a ring $E$ are \emph{equivalent} (we denote by $\alpha \simeq \varepsilon$) if there is a collection of ring isomorphisms $\{\varphi_e : A_e \to E_e : e \in \G_0\}$ such that
\begin{enumerate}
    \item[(i)] $\varphi_{r(g)}(A_g) = E_g$, for all $g \in \G$, 

    \item[(ii)] $\varphi_{r(g)} \circ \alpha_g(a) = \varepsilon_g \circ \varphi_{d(g)}(a)$, for all $a \in A_{g^{-1}}$.
\end{enumerate}
\end{defi}

Let $\alpha = (A_g, \alpha_g)_{g \in \G}$ be a preunital partial action (not necessarily orthogonal) of a finite groupoid $\G$ on a ring $A$. We wish to construct a partial action $\varepsilon$ of $\G$ on a ring $E$ which is orthogonal and equivalent to $\alpha$. Indeed, consider $E = \prod_{e \in \G_0} A_e$ and let $\varphi_e : A_e \to E$ be the natural inclusion given by $\varphi(a) = (a\delta_{e,f})_{f \in \G_0}$. Define $E_g = \varphi_{r(g)}(A_g)$ and $\varepsilon_g = \varphi_{r(g)} \circ \alpha_g \circ \varphi_{d(g)}^{-1}|_{E_{g^{-1}}}$, for all $g \in \G$. Then it is easy to see that $\varepsilon = (E_g,\varepsilon_g)_{g \in \G}$ is an orthogonal partial action of $\G$ on $E$ and $\varepsilon \simeq \alpha$. 

\begin{defi}
    If an orthogonal partial action $\varepsilon$ is equivalent to a partial action $\alpha$, we say that $\varepsilon$ is the \emph{orthogonalization} of $\alpha$.
\end{defi} 

\begin{obs} Observe that all  the orthogonalizations are equivalent, that is, orthogonalization is unique up to equivalence, and that is why we use the term ``the" orthogonalization instead of ``an" orthogonalization.\end{obs}

As well as in the previous section, given a unital partial action $\alpha$, we fix the notation $\beta = (B_g,\beta_g)_{g \in \G}$ to represent its globalization, which is an orthogonal global action on a ring $B = \prod_{e \in \G_0} B_e$ such that $A_e$ is an ideal of $B_e$, for all $e \in \G_0$. We also fix the notation $\varepsilon = (E_g,\varepsilon_g)_{g \in \G}$ to denote the orthogonalization of $\alpha$.

\begin{prop}
    Let $\alpha$ be a unital partial action and $\beta$ its globalization. Then $\beta$ is also the globalization of $\varepsilon$.
\end{prop}
\begin{proof}
    It only takes to observe that we can realize $E = \prod_{e \in \G_0} A_e$ as an ideal of $B$ and that $\varepsilon$ coincides with the standard restriction of $\beta$ to $E$ as defined in \cite[Section 1]{bagio2012partial}, that is, $E_e = B_e \cap E$, for all $e \in \G_0$, $E_g = E_{r(g)} \cap \beta_g(E_{d(g)})$, for all $g \in \G$, and $\varepsilon_g = \beta_g|_{E_{g^{-1}}}$.
\end{proof}

Partial Galois extensions also have some properties preserved by orthogonalizations.

\begin{prop} \label{propalphaglaoisgammagalois}
    On the above notations, assume that $A|_{A^{\alpha}}$ is a partial $\alpha$-Galois extension. Then $E|_{E^{\varepsilon}}$ is a partial $\varepsilon$-Galois extension.
\end{prop}
\begin{proof}
    Let $\{x_i,y_i\}_{i=1}^n$ be an $\alpha$-Galois partial coordinate system of $A$ over $A^{\alpha}$, that is,
    \begin{align*}
        \sum_{i=1}^n x_i\alpha_g(y_i1_{g^{-1}}) = \sum_{e \in \G_0} 1_e\delta_{e,g},
    \end{align*}
    for all $g \in \G$. Write $x_i' = \sum_{e \in \G_0} \varphi_e(x_i1_e) \in E$ and $y_i' = \sum_{e \in \G_0} \varphi_e(y_i1_e) \in E$, for all $1 \leq i \leq n$. Then \begin{align*}
        \sum_{i=1}^n x_i'\varepsilon_g(y_i'\varphi_{d(g)}(1_{g^{-1}})) & = \sum_{i=1}^n \sum_{e \in \G_0} \sum_{f \in \G_0} \varphi_e(x_i1_e)\varepsilon_g(\varphi_f(y_i1_f)\varphi_{d(g)}(1_{g^{-1}})) \\
        & = \sum_{i=1}^n \varphi_{r(g)}(x_i1_{r(g)})\varepsilon_g(\varphi_{d(g)}(y_i1_{d(g)})\varphi_{d(g)}(1_{g^{-1}})) \\
        & = \sum_{i=1}^n \varphi_{r(g)}(x_i1_{r(g)})\varphi_{r(g)} \circ \alpha_g \circ \varphi_{d(g)}^{-1}(\varphi_{d(g)}(y_i1_{g^{-1}})) \\
        & = \sum_{i=1}^n \varphi_{r(g)}(x_i \alpha_g (y_i1_{g^{-1}})1_{r(g)}) = \varphi_{r(g)} \left ( \sum_{i=1}^n x_i\alpha_g (y_i1_{g^{-1}})1_{r(g)} \right ).
    \end{align*}

    There are two cases to consider: (1) $g \in \G_0$; or (2) $g \notin \G_0$. Suppose (1). Then $\varphi_{r(g)}\left ( \sum_{i=1}^n x_i\alpha_g (y_i1_{g^{-1}})1_{r(g)} \right ) = \varphi_{r(g)}(1_{r(g)}) = 1_{C_{r(g)}}$. Now suppose (2). Thus $$\varphi_{r(g)}\left ( \sum_{i=1}^n x_i\alpha_g (y_i1_{g^{-1}})1_{r(g)} \right ) = \varphi_{r(g)}(0) = 0.$$ 
    
    Hence, $\{x_i',y_i'\}_{i=1}^n$ is a $\varepsilon$-Galois partial coordinate system of $E$ over $E^\varepsilon$.
\end{proof}

The converse of the above proposition is not true even for the global case, as we can see in the next example.

\begin{exe} \label{exenaovolta}
Consider $\G = \{g,g^{-1},r(g),d(g)\}$ and $A = Re_1 \oplus Re_2 \oplus Re_3$, where $R \neq 0$ is a commutative ring and the $e_i$'s are pairwise orthogonal, central idempotents such that $e_1 + e_2 + e_3 = 1_A$.

Define $\beta = (A_s,\beta_s)_{s \in \G}$ as \begin{align*}
    A_{g^{-1}} = A_{d(g)} = Re_1 \oplus Re_2, & \quad A_g = A_{r(g)} = Re_2 \oplus Re_3, \\ \beta_g(ae_1 + be_2)  = be_2 + ae_3, & \quad
    \beta_{e} = \text{Id}_{A_e},  \text{ for } e \in \G_0, \,\, \text{ and }
    \beta_{g^{-1}}  = \beta_g^{-1}.
\end{align*}

Hence $ A^{\alpha} = R(e_1 + e_3) \oplus Re_2.$

\noindent \textbf{Claim.} $A$ is not a $\beta$-Galois extension of $A^{\beta}$. 

Indeed, suppose that $A$ is a $\beta$-Galois extension of $A^{\beta}$, that is, there are $x_i = x_{i1}e_1 + x_{i_2}e_2 + x_{i_3}e_3$ and $y_i = y_{i1}e_1 + y_{i_2}e_2 + y_{i_3}e_3$ in $A$ and a positive integer $n$ such that $\sum_{i=1}^n x_i\beta_s(y_i1_{s^{-1}}) = \sum_{e \in \G_0} \delta_{e,s}1_e, \text{ for all } s \in \G.$

In particular, for $s = g$, $\sum_{i=1}^n x_i\beta_g(y_i1_{g^{-1}}) = 0.$ Therefore\begin{align*}
    & \sum_{i=1}^n (x_{i1}e_1 + x_{i2}e_2 + x_{i3}e_3)\beta_g(y_{i1}e_1 + y_{i2}e_2) = 0 \\
    \implies & \sum_{i=1}^n (x_{i1}e_1 + x_{i2}e_2 + x_{i3}e_3)(y_{i1}e_3 + y_{i2}e_2) = 0 \\
    \implies & \sum_{i=1}^n x_{i2}y_{i2}e_2 + x_{i3}y_{i1}e_3 = 0 \\
    \implies & \left ( \sum_{i=1}^n  x_{i2}y_{i2} \right )e_2 + \left ( \sum_{i=1}^n x_{i3}y_{i1} \right ) e_3 = 0.
\end{align*}
Then $\sum_{i=1}^n x_{i2}y_{i2} = 0 = \sum_{i=1}^n x_{i3}y_{i1}$. On the other hand, since $\sum_{i=1}^n x_i\beta_{r(g)}(y_i1_{g^{-1}}) = e_2 + e_3,$ it follows that
\begin{align*}
    & \sum_{i=1}^n (x_{i1}e_1 + x_{i2}e_2 + x_{i3}e_3)\beta_{r(g)}(y_{i2}e_2 + y_{i3}e_3) = e_2 + e_3 \\
    \implies & \sum_{i=1}^n (x_{i1}e_1 + x_{i2}e_2 + x_{i3}e_3)(y_{i2}e_2 + y_{i3}e_3) = e_2 + e_3 \\
    \implies & \sum_{i=1}^n x_{i2}y_{i2}e_2 + x_{i3}y_{i3}e_3 = e_2 + e_3\\
    \implies & \left ( \sum_{i=1}^n  x_{i2}y_{i2} \right )e_2 + \left ( \sum_{i=1}^n  x_{i3}y_{i3} \right )e_3 = e_2 + e_3.
\end{align*}
Hence $\sum_{i=1}^n x_{i2}y_{i2} = 1_R = \sum_{i=1}^n x_{i3}y_{i1}$. But then $1_R = 0$, a contradiction. 

Now consider the orthogonalization $\varepsilon$ of $\alpha$ given as follows: $E = \prod_{e \in \G_0} A_e = A_{d(g)} \times A_{r(g)} = (Re_1 \oplus Re_2) \times (Re_2 \oplus Re_3)$, $E_{g^{-1}} = E_{d(g)} = (Re_1 \oplus Re_2,0)$, $E_g = E_{r(g)} = (0, Re_2 \oplus Re_3)$, $\varepsilon_{r(g)} = \text{Id}_{E_g}$, $\varepsilon_{d(g)} = \text{Id}_{E_{g^{-1}}}$, $\varepsilon_g((ae_1 + be_2, 0)) = (0, ae_2 + be_3)$ and $\varepsilon_{g^{-1}} = \varepsilon_g^{-1}$.  We can easily see that $E$ is $\varepsilon$-Galois over $E^{\varepsilon} = R(e_1,e_2) \oplus R(e_2,e_3)$ with coordinate system $\{x_1 = y_1 = (e_1,0), x_2 = y_2 = (e_2,0), x_3 = y_3 = (0,e_2), x_4 = y_4 = (0,e_3)\}$.
\end{exe}

As a corollary of Proposition \ref{propalphaglaoisgammagalois} we have the following.

\begin{corollary} \label{corgaloisparcialglobalnaoortog}
    Assume that $\alpha$ is a unital partial action (not necessarily orthogonal). If $A$ is $\alpha$-Galois partial over $A^{\alpha}$, then $B$ is $\beta$-Galois over $B^{\beta}$.
\end{corollary}
\begin{proof}
    By Proposition \ref{propalphaglaoisgammagalois}, $A|_{A^{\alpha}}$ partial $\alpha$-Galois implies $E|_{E^{\varepsilon}}$ partial $\varepsilon$-Galois. By Lemma \ref{lemagaloisparcialsseglobal}, $E|_{E^{\varepsilon}}$ partial $\varepsilon$-Galois implies $B|_{B^{\beta}}$ $\beta$-Galois, which concludes the proof.
\end{proof}   

\subsection{Strongly $\alpha$-Galois Extension}

Keeping the notations of the previous subsection, define \begin{align*}
    \varphi : A & \to E \\
    a & \mapsto  (\varphi_e(a1_e))_{e \in \G_0}.
\end{align*}

We have that $\varphi$ is a monomorphism. Indeed, $\varphi(a) = 0$ if and only if $\varphi_e(a1_e) = 0$, for all $e \in \G_0$. But $\varphi_e(a1_e) = 0$ if and only if $a1_e = 0$, since $\varphi_e : A_e \to E_e$ is a monomorphism, and from (P1) it follows that $a = 0$ if and only if $a1_e = 0$, for all $e \in \G_0$.

\begin{prop} \label{proptrad}
    Let $\cH$ be a wide subgroupoid of $\G$. The following assertion holds: $\varphi(A^{\alpha|_{\cH}}) = E^{\varepsilon|_{\cH}} \cap \varphi(A)$.
\end{prop}
\begin{proof}
    Take $a \in A^{\alpha|_{\cH}}$. Given $h \in \cH$, 
    \begin{align*}
        \varepsilon_h(\varphi(a)\varphi_{d(h)}(1_{h^{-1}})) & = \varepsilon_h(\varphi_{d(h)}(a1_{d(h)})\varphi_{d(h)}(1_{h^{-1}})) = \varepsilon_h(\varphi_{d(h)}(a1_{d(h)}1_{h^{-1}})) \\
        & = \varepsilon_h(\varphi_{d(h)}(a1_{h^{-1}})) = \varphi_{r(h)} \circ \alpha_h \circ \varphi_{d(h)}^{-1}(\varphi_{d(h)}(a1_{h^{-1}})) \\
        & = \varphi_{r(h)} \circ \alpha_h(a1_{h^{-1}}) = \varphi_{r(h)}(a1_{h}) \\
        & = \varphi_{r(h)}(a)\varphi_{r(h)}(1_{h}) = \varphi(a)\varphi_{r(h)}(1_{h}).
    \end{align*}
    
    Since we chose $h$ arbitrarily, we have that $\varphi(a) \in E^{\varepsilon|_{\cH}}$, proving the first inclusion.

    For the reverse, consider $a \in A$ such that $\varphi(a) \in E^{\varepsilon|_{\cH}}$. Then, given $h \in \cH$, we have that $\varepsilon_h(\varphi(a)\varphi_{d(h)}(1_{h^{-1}})) = \varphi(a)\varphi_{r(h)}(1_h)$. However
    \begin{align*}
        \varepsilon_h(\varphi(a)\varphi_{d(h)}(1_{h^{-1}})) & = \varepsilon_h(\varphi_{d(h)}(a1_{d(h)})\varphi_{d(h)}(1_{h^{-1}})) = \varepsilon_h(\varphi_{d(h)}(a1_{h^{-1}})) = \varphi_{r(h)} \circ \alpha_h(a1_{h^{-1}}),
    \end{align*}
    and $\varphi(a)\varphi_{r(h)}(1_h) = \varphi_{r(h)}(a1_{r(h)})\varphi_{r(h)}(1_h) = \varphi_r(h)(a1_h)$.

    Therefore, $\varphi_{r(h)}(\alpha_h(a1_{h^{-1}})) = \varphi_{r(h)}(a1_h)$. Since $\varphi_{r(h)}$ is a monomorphism, it follows that $\alpha_h(a1_{h^{-1}}) = a1_h$, for all $h \in \cH$, that is, $a \in A^{\alpha|_{\cH}}$.
\end{proof}

\begin{defi} \label{defstrongGalois} Let $\alpha$ be a partial action of the groupoid $\G$ on a ring $A$. Consider the orthogonalization $\varepsilon$ of $\alpha$ on a ring $E$. We say that $\alpha$ is a \emph{strongly Galois partial action} if:
    \begin{enumerate}
        \item[(i)] $A_g \neq 0$, for all $g \in \G_0$;

        \item[(ii)] $A$ is a partial $\alpha$-Galois extension of $A^{\alpha}$;

        \item[(iii)] For all wide subgroupoids $\cH_1 \neq \cH_2$ of $\G$, it holds that $$E^{\varepsilon|_{\cH_1}} \cap \varphi(A) \neq E^{\varepsilon|_{\cH_2}} \cap \varphi(A).$$
    \end{enumerate} In this case, we also say that $A$ is a \emph{strongly $\alpha$-Galois extension over} $A^{\alpha}$.
\end{defi}

Now we are able to state a Galois correspondence for the nonorthogonal partial case. Firstly, we set:
    \begin{align*}
       \mathsf{W} & = \{ \cH \subseteq \G : \cH \text{ is a wide subgroupoid of } \G \}, \\
        \mathsf{A} & = \{ C \subseteq A : C = \varphi^{-1}(E^{\varepsilon|_{\cH}} \cap \varphi(A)), \text{ for some } \cH \in  \mathsf{W}\}.
    \end{align*}

\begin{prop} Assume $A|_{A^{\alpha}}$ a partial $\alpha$-Galois extension. If $C \in \mathsf{A}$, then $A|_{C}$ is a partial $\alpha|_{\cH}$-Galois extension for some wide subgroupoid $\cH$ of $\G$.\end{prop}
\begin{proof}
    If $C \in \mathsf{A}$, then $C = \varphi^{-1}(E^{\varepsilon|_{\cH}} \cap \varphi(A))$, for some $\cH \in \mathsf{W}$. By Proposition \ref{proptrad}, $C = \varphi^{-1}(\varphi(A^{\alpha|_{\cH}})) = A^{\alpha|_{\cH}}$. Now, consider a partial $\alpha$-Galois coordinate system $\{x_i,y_i\}_{i=1}^n$ for $A$ over $A^{\alpha}$, that is, 
    \begin{align*}
        \sum_{i=1}^n x_i\alpha_g(y_i1_{g^{-1}}) = \sum_{e \in \G_0} \delta_{e,g}1_e,
    \end{align*}
    for all $g \in \G$. Then it is clear that
    \begin{align*}
        \sum_{i=1}^n x_i\alpha_h(y_i1_{h^{-1}}) = \sum_{e \in \G_0} \delta_{e,h}1_e,
    \end{align*}
    for all $h \in \cH$, that is, $A$ is a partial $\alpha|_{\cH}$-Galois extension of $A^{\alpha|_{\cH}} = C$.\end{proof}

Now we can prove the main result of this subsection. From now on, assume that the ring $A$ is commutative.

\begin{theorem}[Galois Correspondence for Strongly Galois Partial Groupoid Actions] \label{teogalnaoort}
    Let $\G$ be a finite groupoid acting partially on a commutative ring $A$ via a unital strongly Galois partial action $\alpha$. Then there is a one-to-one correspondence between the sets $ \mathsf{W}$ and $ \mathsf{A}$ given by $ \mathsf{W} \ni \cH \mapsto A^{\alpha|_{\cH}}$ with inverse $ \mathsf{A} \ni C \mapsto \G_{\varphi(C)}$.
\end{theorem}
\begin{proof}
    First apply Theorem \ref{teofund} on the orthogonalization $\varepsilon$ of $\alpha$, so that we obtain a Galois correspondence between the wide subgroupoids of $\G$ and the $E^\varepsilon$-separable, $\varepsilon$-strong subalgebras $T$ of $E$ whose $\G_T$ is a wide subgroupoid. By Proposition \ref{proptrad} and the strongly Galois hypothesis, there is a one-to-one correspondence between those subalgebras $T$ and the $A^\alpha$-subalgebras $C$ that appear in the set $\mathsf{A}$, via $T \leftrightarrow T \cap \varphi(A) = C$. Composing these two correspondences we obtain the one stated in the theorem, finishing the proof.
\end{proof}

In Definition \ref{defstrongGalois}, clearly if $\G$ is a group or if $\alpha$ is already orthogonal, the condition (iii) holds. But these are not the only classes that satisfy (i)-(iii). In the following we can see an example.

\begin{exe}
    Let $G= \{1_G,g,h,gh\}$ be a subgroup of $\mathbb{S}_8$, where
    \begin{align*}
        g = \binom{1 \; 2 \; 3 \; 4 \; 5 \; 6 \; 7 \; 8}{2 \; 1 \; 4 \; 3 \; 6 \; 5 \; 8 \; 7}, \quad h = \binom{1 \; 2 \; 3 \; 4 \; 5 \; 6 \; 7 \; 8}{3 \; 4 \; 1 \; 2 \; 7 \; 8 \; 5 \; 6}, \quad gh = \binom{1 \; 2 \; 3 \; 4 \; 5 \; 6 \; 7 \; 8}{4 \; 3 \; 2 \; 1 \; 8 \; 7 \; 6 \; 5}.
    \end{align*}

    Consider $\G = \mathcal{A}_2 \times G$ with $\G_0 = \{f_1,f_2\}$, that is, $G = \{(f_1,f_1,1_G),(f_1,f_1,g),(f_1,f_1,h), $ $(f_1,f_1,gh),(f_1,f_2,1_G),(f_1,f_2,g),(f_1,f_2,h),(f_1,f_2,gh),(f_2,f_1,1_G),(f_2,f_1,g),$ $(f_2,f_1,h), \qquad$ $ (f_2,f_1,gh),$ $(f_2,f_2,1_G),$ $(f_2,f_2,g),(f_2,f_2,h),(f_2,f_2,gh)\}$. Given $k = (i,j,\ell) \in \G$ with $i,j \in \{f_1,f_2\}$, $\ell \in G$, we define $\partial_k = \ell$.

    Consider a commutative ring $R$ and the $R$-algebra $A = \bigoplus_{i=1}^{12} Re_i$, where $\{e_i\}_{i=1}^{12}$ is a set of pairwise orthogonal idempotents whose sum is $1_A$. Defining $A_{k_1} = \bigoplus_{i=1}^8 Re_i$, $A_{k_2} = \bigoplus_{i=5}^{12} Re_i$, for all $k_1, k_2 \in \G$ with $r(k_1) = f_1$, $r(k_2) = f_2$, and
    \begin{align*}
        \beta_{k}\left (\sum_{i=1}^8 a_ie_i \right) & = \begin{cases}
            \sum_{i=1}^8 a_ie_{\partial_k(i)}, \text{ if } d(k) = r(k) = f_1, \\
            \sum_{i=1}^8 a_ie_{\partial_k(i)+4}, \text{ if } d(k) = f_1,  r(k) = f_2,
        \end{cases} \\
        \beta_{k}\left (\sum_{i=1}^8 a_ie_{i+4} \right) & = \begin{cases}
            \sum_{i=1}^8 a_ie_{\partial_k(i)+4}, \text{ if } d(k) = r(k) = f_2, \\
            \sum_{i=1}^8 a_ie_{\partial_k(i)}, \text{ if } d(k) = f_2,  r(k) = f_1,
        \end{cases}
    \end{align*}
    we have that $\beta = (A_k,\beta_k)_{k \in \G}$ is a nonorthogonal action of $\G$ on $A$.
    
    Moreover, $A$ is strongly $\beta$-Galois over $A^\beta = R$ with coordinate system $\{x_i = y_i = e_i\}_{i=1}^{12}$. We will present the Galois correspondence for this action. For that, we will denote $\G(f_1) = G_1$, $\G(f_2) = G_2$, and we will use the following bracket notation to describe the subalgebras: we will represent by $[i_1, i_2, \ldots, i_n]$ the $R$-algebra $R(e_{i_1} + e_{i_2} \cdots e_{i_n}) \oplus \bigoplus_{j \notin I} Re_j$, where $I = \{i_1, \ldots, i_n\} \subseteq \{1, \ldots, 12\}$, and we will omit the direct sum symbol. For example,
    \begin{align*}
        [1,2] & \leftrightarrow R(e_1 + e_2) \oplus \bigoplus_{i=3}^{12} Re_i, \\
        [3,4][5,6,7] & \leftrightarrow Re_1 \oplus Re_2 \oplus R(e_3 + e_4) \oplus R(e_5 + e_6 + e_7) \oplus \bigoplus_{i=8}^{12} Re_i, \\
        [1,2,3,4,5,6,7,8,9,10,11,12] & \leftrightarrow A^\beta.
    \end{align*}

    The Galois correspondence relates the 33 elements of $ \mathsf{W}$ with the 33 elements of $ \mathsf{A}$:
    
    \begin{align*}
        \G_0 & \leftrightarrow A \\
        \G_0 \cup \{(f_1,f_1,g)\} & \leftrightarrow [1,2][3,4][5,6][7,8] \\
        \G_0 \cup \{(f_1,f_1,h)\} & \leftrightarrow [1,3][2,4][5,7][6,8] \\
        \G_0 \cup \{(f_1,f_1,gh)\} & \leftrightarrow [1,4][2,3][5,8][6,7]\\
        \G_0 \cup G_1 & \leftrightarrow [1,2,3,4][5,6,7,8]\\
        \G_0 \cup \{(f_2,f_2,g)\} & \leftrightarrow [5,6][7,8][9,10][11,12] \\
        \G_0 \cup \{(f_2,f_2,h)\} & \leftrightarrow [5,7][6,8][9,11][10,12] \\
        \G_0 \cup \{(f_2,f_2,gh)\} & \leftrightarrow [5,8][6,7][9,12][10,11] \\
        \G_0 \cup G_2 & \leftrightarrow [5,6,7,8][9,10,11,12] \\ 
        H_1 := \G_0 \cup \{(f_1,f_1,g),(f_2,f_2,g)\} & \leftrightarrow [1,2][3,4][5,6][7,8][9,10][11,12] \\ 
        \G_0 \cup \{(f_1,f_1,g),(f_2,f_2,h)\} & \leftrightarrow [1,2][3,4][5,6,7,8][9,11][10,12] \\ 
        \G_0 \cup \{(f_1,f_1,g),(f_2,f_2,gh)\} & \leftrightarrow [1,2][3,4][5,6,7,8][9,12][10,11] \\
         \G_0 \cup \{(f_1,f_1,g)\} \cup G_2 & \leftrightarrow [1,2][3,4][5,6,7,8][9,10,11,12] \\ 
        \G_0 \cup \{(f_1,f_1,h),(f_2,f_2,g)\} & \leftrightarrow [1,3][2,4][5,6,7,8][9,10][11,12] \\ 
        H_2 := \G_0 \cup \{(f_1,f_1,h),(f_2,f_2,h)\} & \leftrightarrow [1,3][2,4][5,7][6,8][9,11][10,12]\\
        \G_0 \cup \{(f_1,f_1,h),(f_2,f_2,gh)\} & \leftrightarrow [1,3][2,4][5,6,7,8][9,12][10,11] \\ 
        \G_0 \cup \{(f_1,f_1,h)\} \cup G_2 & \leftrightarrow [1,3][2,4][5,6,7,8][9,10,11,12] \\ 
        \G_0 \cup \{(f_1,f_1,gh),(f_2,f_2,g)\} & \leftrightarrow [1,4][2,3][5,6,7,8][9,10][11,12] \\ 
        \G_0 \cup \{(f_1,f_1,gh),(f_2,f_2,h)\} & \leftrightarrow [1,4][2,3][5,6,7,8][9,11][10,12] \\ 
        H_3 := \G_0 \cup \{(f_1,f_1,gh),(f_2,f_2,gh)\} & \leftrightarrow [1,4][2,3][5,8][6,7][9,12][10,11] \\
        \G_0 \cup \{(f_1,f_1,gh)\} \cup G_2 & \leftrightarrow [1,4][2,3][5,6,7,8][9,10,11,12]
       \end{align*}

        \begin{align*}
        \G_0 \cup G_1 \cup \{(f_2,f_2,g)\} & \leftrightarrow [1,2,3,4][5,6,7,8][9,10][11,12] \\
        \G_0 \cup G_1 \cup \{(f_2,f_2,h)\} & \leftrightarrow [1,2,3,4][5,6,7,8][9,11][10,12] \\
        \G_0 \cup G_1 \cup \{(f_2,f_2,gh)\} & \leftrightarrow [1,2,3,4][5,6,7,8][9,12][10,11] \\
        \G_0 \cup G_1 \cup G_2 & \leftrightarrow [1,2,3,4][5,6,7,8][9,10,11,12] \\
        K_0 := \G_0 \cup \{(f_1,f_2,e),(f_2,f_1,e)\} & \leftrightarrow [1,5,9][2,6,10][3,7,11][4,8,12] \\
        K_1 := \G_0 \cup \{(f_1,f_2,g),(f_2,f_1,g)\} & \leftrightarrow [1,6,9][2,5,10][3,8,11][4,7,12] \\
        K_2 := \G_0 \cup \{(f_1,f_2,h),(f_2,f_1,h)\} & \leftrightarrow [1,7,9][2,8,10][3,5,11][4,6,12] \\
        K_3 := \G_0 \cup \{(f_2,f_1,gh),(f_2,f_1,gh)\} & \leftrightarrow [1,8,9][2,7,10][3,6,11][4,5,12] \\
        K_0 \cup K_1 \cup H_1 & \leftrightarrow [1,2,5,6,9,10][3,4,7,8,11,12] \\
        K_0 \cup K_2 \cup H_2 & \leftrightarrow [1,3,5,7,9,11][2,4,6,8,10,12] \\
        K_0 \cup K_3 \cup H_3 & \leftrightarrow [1,4,5,8,9,12][2,3,6,7,10,11] \\
        \G & \leftrightarrow A^\beta
    \end{align*}
\end{exe}

\begin{obs}Until now, we can conclude the following:
\begin{enumerate}
    \item[(1)] There are properties in Galois theory that are valid for orthogonal actions, but not for nonorthogonal ones, so the theory cannot be reduced to the orthogonal case.
    \item[(2)] Theorem \ref{teogalnaoort} gives a generalization of the Galois correspondences that already exist in the literature, but it is still restrictive, once it requires that the partial action is strongly Galois, and not all partial actions over Galois extensions belongs to this class, as we can see in the next example.
\end{enumerate}
 \end{obs}

\begin{exe}
Let $\G = \{g,f_1\} \cup \{h,f_2\}$ be the groupoid given by the disjoint union of two groups isomorphic to $\mathbb{Z}_2$, that is, $g^2 = f_1 = f_1^2$ and $h^2 = f_2 = f_2^2$. Consider a commutative ring $R$ and $A = Re_1 \oplus Re_2$, with $e_1, e_2$ orthogonal idempotents whose sum is $1_A$. Define $A_{k} = A$, for all $k \in \G$,
\begin{align*}
    \beta_g(ae_1 + be_2) = be_1 + ae_2 = \beta_h(ae_1 + be_2),
\end{align*}
and $\beta_{f_1} = \beta_{f_2} = \text{id}_A$. We have that $\beta$ is a (nonorthogonal) global action of $\G$ on $A$. Moreover, $A$ is $\beta$-Galois over $A^{\beta} = R$ with coordinate system $x_i = y_i = e_i$, $i = 1,2$. However, we have that $A^{\beta|_{\G(f_1)}} = A^{\beta|_{\G(f_2)}} = A^{\beta}$, from where it follows that $A$ is not strongly $\beta$-Galois over $A^{\beta}$.
\end{exe}

\subsection{Weakening the hypothesis}

To finish the paper, we extend, in the global case, the Galois correspondence for Galois extensions that are not strongly Galois. The main argument is the following proposition.

\begin{prop} \label{propglobalmaximal}
    Let $\beta = (A_g,\beta_g)$ be a global action of a groupoid $\G$ on a commutative ring $A$. If $\cH_1$ and $\cH_2$ are wide groupoids of $\G$ such that $A^{\beta|_{\cH_1}} = A^{\beta|_{\cH_2}}$, then the wide subgroupoid of $\G$ generated by $\cH_1$ and $\cH_2$, named $\cH$, is such that $A^{\beta|_{\cH}} = A^{\beta|_{\cH_1}}$.
\end{prop}
\begin{proof}
    It is true that $A^{\beta|_\cH} \subseteq A^{\beta|_{\cH_1}}$, since $\cH_1 \subseteq \cH$. For the reverse inclusion, consider $a \in A^{\beta|_{\cH_1}} = A^{\beta|_{\cH_2}}$. Take $h \in \cH$. We have that $h = h_1 \cdots h_n$, for some positive integer $n$ and $h_1, \ldots, h_n \in \cH_1 \cup \cH_2$ with $d(h_i) = r(h_{i+1})$, for all $1 \leq i \leq n-1$. Then
    \begin{align*}
        \beta_h(a1_{h^{-1}}) & = \beta_{h_1 \cdots h_n}(a1_{d(h_n)})  = \beta_{h_1 \cdots h_{n-1}}(\beta_{h_n}(a1_{d(h_n)})) \\
        & = \beta_{h_1 \cdots h_{n-1}}(a1_{r(h_n)}) 
         = \beta_{h_1 \cdots h_{n-1}}(a1_{d(h_{n-1})}) \\
        & = \beta_{h_1 \cdots h_{n-2}}(\beta_{h_{n-1}}(a1_{d(h_{n-1})})) 
        = \cdots \\
        & = \beta_{h_1}(a1_{d(h_1)}) = a1_{r(h_1)} = a1_{h}.
    \end{align*}
\end{proof}

Keeping the notations of the previous subsection, we define an equivalence relation $\sim$ in $ \mathsf{W}$ by: $\cH_1 \sim \cH_2$ if and only if $A^{\alpha|_{\cH_1}} = A^{\alpha|_{\cH_2}}$. If $\alpha$ is global, then by Proposition \ref{propglobalmaximal} each $\sim$-class has a unique maximal element with respect to inclusion. Namely, given $\cH \in  \mathsf{W}$, we write $\cH^{max}$ to represent the maximal element of the $\sim$-class of $\cH$. Therefore we can define the set
\begin{align*}
     \mathsf{W}^{max} = \{ \cH^{max} : \cH \in  \mathsf{W} \}.
\end{align*}

Now we state the correspondence that finish our paper. We keep the same notation $\mathsf{W}$ and $\mathsf{W}^{max}$  for a global action $\beta$ on a ring $A$. When the extension $A|_{A^\beta}$ is strongly Galois in the theorem below, we have exactly Theorem \ref{teogalnaoort} for global actions, since in this case every $\sim$-class consists of only one wide subgroupoid. If $\beta$ is orthogonal, we have the Galois correspondence for orthogonal groupoid actions on commutative rings \cite[Theorem 4.6]{paques2018galois}, since $\beta$ being orthogonal implies that the extension $A|_{A^\beta}$ is strongly Galois and that $\beta$ is equal to its own orthogonalization.

\begin{theorem}[Galois Correspondence for Global Groupoid Actions]\label{teofinal}
    Let $\G$ be a finite groupoid acting on a commutative ring $A$ via a unital global action $\beta$. Consider the orthogonalization $\varepsilon$ of $\beta$ on a ring $E$. Assume that $A_g \neq 0$, for all $g \in \G_0$, and that $A$ is a $\beta$-Galois extension of $A^{\beta}$. Then there is a one-to-one correspondence between the sets $ \mathsf{W}^{max}$ and $ \mathsf{A}$ given by $ \mathsf{W}^{max} \ni \cH \mapsto A^{\beta|_{\cH}}$ with inverse $ \mathsf{A} \ni C \mapsto (\G_{\varphi(C)})^{max}$.
\end{theorem}
\begin{proof}
Firstly we apply \cite[Theorem 4.6]{paques2018galois} on the orthogonalization $\varepsilon$ of $\beta$ to obtain a Galois correspondence between all the wide subgroupoids of $\G$ and all the $E^\varepsilon$-separable, $\varepsilon$-strong subalgebras $C$ of $E$, that is, $\cH \leftrightarrow E^{\varepsilon|_{\cH}}$. Restricting this correspondence to $\mathsf{W}^{max}$, we have a one-to-one correspondence $\cH^{max} \leftrightarrow E^{\varepsilon|_{\cH^{max}}}$.

\noindent \emph{Claim 1.} There is a one-to-one correspondence $E^{\varepsilon|_{\cH^{max}}} \leftrightarrow A^{\beta|_{\cH^{max}}} = \varphi^{-1}(E^{\varepsilon|_{\cH^{max}}} \cap \varphi(A))$.

Indeed, define $\theta : \{E^{\varepsilon|_{\cH^{max}}} : \cH \in \mathsf{W}\} \to \{A^{\beta|_{\cH^{\max}}} : \cH \in \mathsf{W}\}$ by $\theta(E^{\varepsilon|_{\cH^{max}}}) = A^{\beta|_{\cH^{\max}}}$. We have that $\theta$ is injective, since if $\theta(E^{\varepsilon|_{\cH^{max}}}) = \theta(E^{\varepsilon|_{\mathcal{K}^{max}}})$, for $\cH^{max}, \mathcal{K}^{max} \in \mathsf{W}^{max}$, then $A^{\beta|_{\cH^{max}}} = A^{\beta|_{\mathcal{K}^{max}}}$, from where it follows by Proposition \ref{propglobalmaximal} that $\cH \sim \mathcal{K}$, that is, $\cH^{max} = \mathcal{K}^{max}$. Hence $E^{\varepsilon|_{\cH^{max}}} = E^{\varepsilon|_{\mathcal{K}^{max}}}$. Furthermore, Proposition \ref{proptrad} implies that $\theta$ is surjective, proving the claim.

\noindent \emph{Claim 2.} $\mathsf{A} = \{A^{\beta|_{\cH^{max}}}: \cH^{max} \in  \mathsf{W}^{max}\}$.

Indeed, for each $C \in \mathsf{A}$, there is a wide subgroupoid $\cH \in \mathsf{W}$ such that $C = A^{\beta|_{\cH}}$. But $\cH \sim \cH^{max}$, that is, $A^{\beta|_{\cH}} = A^{\beta|_{\cH^{max}}}$, by Proposition \ref{propglobalmaximal}. So $\mathsf{A} \subseteq \{A^{\beta|_{\cH^{max}}}: \cH^{max} \in  \mathsf{W}^{max}\}$. The other inclusion is immediate.

So, by Claim 1 and Claim 2, there is a one-to-one correspondence between the sets $ \mathsf{W}^{max}$ and $\mathsf{A}$.
\end{proof}

\begin{obs}
    The same argument used in Proposition \ref{propglobalmaximal} cannot be applied in the partial case to obtain an analogous of Theorem \ref{teofinal}, since given $a \in A$, $h,k \in \cH_1 \cup \cH_2$ such that $(h,k) \in \G_2$, we do not have, in general, that $a1_{k^{-1}h^{-1}} \in A_{k^{-1}} \cap A_{k^{-1}h^{-1}}$. Therefore we cannot guarantee that $\alpha_{hk}(a1_{k^{-1}h^{-1}}) = \alpha_h \circ \alpha_k(a1_{k^{-1}h^{-1}})$. The existence of a more general Galois correspondence theorem for partial groupoid actions will require other technique and is still an open question.
\end{obs}


\begin{thebibliography}{1}
\bibitem{bagio2010partial}
D. Bagio, D. Flores, and A. Paques. Partial actions of ordered groupoids on rings. \emph{J. Algebra Appl.}, 9(03):501–517, 2010.

\bibitem{bagio2012partial}
D. Bagio and A. Paques. Partial groupoid actions: globalization, Morita theory, and Galois theory. \emph{Commun. Algebra}, 40(10):3658–3678, 2012.

\bibitem{bagio2020restriction}
D. Bagio, A. Paques, and H. Pinedo. Restriction and extension of partial actions. \emph{J. Pure Appl. Algebra}, 224(10):106391, 2020.

\bibitem{bagio2022galois}
D. Bagio, A. Sant’Ana, and T. Tamusiunas. Galois correspondence for group-type partial actions of groupoids. \emph{Bull. Belgian Math. Soc.-Simon Stevin}, 28(5):745–767, 2022.

\bibitem{chase1969galois}
S. U. Chase, D. K. Harrison, and A. Rosenberg. Galois theory and cohomology of commutative rings, volume 52. \emph{Am. Math. Soc.}, 1965.

\bibitem{cortes2017characterisation}
W. Cortes and T. Tamusiunas. A characterisation for a groupoid Galois extension using partial isomorphisms. \emph{Bull. Aust. Math. Soc.}, 96(1):59–68, 2017.

\bibitem{dokuchaev2007partial}
M. Dokuchaev, M. Ferrero, and A. Paques. Partial actions and Galois theory. \emph{J. Pure Appl. Algebra}, 208(1):77–87, 2007.

\bibitem{exel1998partial}
R. Exel. Partial actions of groups and actions of inverse semigroups. \emph{Proc. Am. Math. Soc.}, 126(12):3481–3494, 1998.

\bibitem{garta2024}
C. Garcia and T. Tamusiunas. A Galois correspondence for $K_{\beta}$-rings. \emph{J. Algebra}, 640:74–105, 2024.

\bibitem{knus2006theorie}
M. A. Knus and M. Ojanguren. Théorie de la descente et algèbres d’Azumaya, volume 389. \emph{Springer}, 2006.

\bibitem{lata2021galois}
W. G. Lautenschlaeger and T. Tamusiunas. Maximal ordered groupoids and a Galois correspondence for inverse semigroup orthogonal actions. \emph{Appl. Categ. Struct.}, 31(5):31, 2023.

\bibitem{lau2021semigrupoides}
W. G. Lautenschlaeger and T. R. Tamusiunas. Inverse semigroupoid actions and representations. \emph{REMAT}, 9(1):e3006, 2023.

\bibitem{lawson1998inverse}
M. V. Lawson. Inverse semigroups, the theory of partial symmetries. \emph{World Scientific}, 1998.

\bibitem{pata}
A. Paques and T. Tamusiunas. A Galois-Grothendieck-type correspondence for groupoid actions. \emph{Algebra Discrete Math.}, 17(1):80–97, 2014.

\bibitem{paques2018galois}
A. Paques and T. Tamusiunas. The Galois correspondence theorem for groupoid actions. \emph{J. Algebra}, 509:105–123, 2018.

\bibitem{pataIII}
A. Paques and T. Tamusiunas. On the Galois map for groupoid actions. \emph{Comm. Algebra}, 49(3):1037–1047, 2021.

\bibitem{pedrotti2023injectivity}
J. Pedrotti and T. Tamusiunas. Injectivity of the Galois map. \emph{Bull. Brazilian Math. Soc.}, New Series, 54(1):1–11, 2023.

\bibitem{villamayor1966galois}
O. Villamayor and D. Zelinsky. Galois theory for rings with finitely many idempotents. \emph{Nagoya Math. J.}, 27(2):721–731, 1966.

\end{thebibliography}
\end{document}